\documentclass[a4paper,12pt]{article}
\usepackage[utf8x]{inputenc}
\usepackage{mathtools}
\usepackage{amsmath,amssymb,amsthm}
\usepackage{color}
\usepackage
	[
	colorlinks,
	backref
	]
{hyperref}
\usepackage[all]{xy}
\usepackage{float}

\title{
On Bias and Rank}
\author{ David Kazhdan \and 
Tomer M. Schlank}

\newtheorem{thmvoid}{}[section]

\newtheorem{lemma}[thmvoid]{Lemma}
\newtheorem{theorem}[thmvoid]{Theorem}

\newtheorem{corollary}[thmvoid]{Corollary}
\theoremstyle{definition}
\newtheorem{definition}[thmvoid]{Definition}
\newtheorem{remark}[thmvoid]{Remark}
\newcommand{\comm}[1]{}

\newcommand{\QQ}{\mathbb{Q}}
\newcommand{\CC}{\mathbb{C}}

\newcommand{\mcF}{\mathcal F}
\newcommand{\mF}{\mathbb F}
\begin{document}

\maketitle
\abstract{

Given a hypersurface $X\subset \mathbb{P}^{N+1}_{\mathbb{C}}$ Dimca gave a proof showing that the cohomologies of X are the same as the projective space in a range determined by the dimension of the singular locus of X. We prove the analog of Dimca's result case when $\mathbb{C}$ is replaced with an algebraically closed field of finite characteristic and singular cohomology is replaced with $\ell$-adic \'etale cohomology. The Weil conjectures allow relating results about \'eatle cohomology to counting problems over a finite field. Thus by applying this result, we are able to get a relationship between the algebraic properties of certain polynomials and the size of their zero set.
}
\section{introduction}

Let $X\subset \mathbb{P}^{N+1}_{\CC}$ be a hypersurface. If $X$ is smooth, Lefschetz  hyperplane theorem  implies  that all the cohomologies of $X$ but the one in  dimension $N$ are the same as for the projective  spaces. 
This result has a qualitative generalisation due to Dimca \cite{Dimca1}:

\begin{lemma}
Let $X \subset \mathbb{P}^{N+1}_{\CC}$  be a projective hyper-surface of dimension $n$ such that  the singular locus of $X$ is in at least  in codimention $c$ then:
$$H^m(X(\CC)) = H^m(\mathbb{P}^{N}(\CC))$$ for $2N-(c+2) \leq m \leq 2N$

\end{lemma}

The first main  result of this paper is Lemma ~\ref{lem:A} which is the analog of Dimca's result case when $\CC$ is replaced with an algebraic closed field of finite  characteristic and singular cohomology is replaces with $\ell$-adic \'eatle cohomology. 
The Weil conjectures  allows to relate results about \'eatle cohomology to counting problems over finite field. Thus by applying Lemma \ref{lem:A}  we able to get a relationship between the algebraic properties of certain polynomials and the size of their zero set.

We thank Dennis Gaitsgory and  Michael Finkelberg for useful discussions .
\section{The results}
More specifically  
Let $p$ be a prime and $q = p^m, k= \mathbb{F}_q $. We denote by $k_n$ the extension of $k$ of degree $n$ and by $\bar{k}$ the algebraic closure of  $k$.

Let $V$ be a vector space over $k$ of dimension $N$, and let $F \in k[V]$ be a polynomial.

For every  $n \in \mathbb{N}$,  $F$ defines a map
$$F|_{n} \colon V(k_n) \to  k_n$$

As a map of finite sets $F|_{n}$ induces a distribution on $ k_n$ according to size of the fibres.
In this paper we ask the following:

Assume that for $n$ large,  this  distribution is very far from uniform. Can we deduce a structural result on $F$?
Can one say that it is ``degenerate'' in some sense?
In \cite{Lov}, Bhowmick and Lovett prove  a  beautiful result  along this lines using combinatorial  methods . We give a similar result based on Lemma \ref{lem:A}.   To state our result in a precise way we need first to define a qualitative measurements of both the non-uniformity of the aforementioned distribution and the ``degeneracy'' of $F$. In next two subsections we deal with these two tasks one by one.

\subsubsection{The bias of a polynomial}
In the notation as above let $\nu_n$ be the uniform probability measure on the elements of $ V(K_n) $.
We denote $$
\mu_n^F \coloneqq (F|_n)_*(\nu_n)$$
As the push-forward probability measure.  
Simply put we have for  every $t \in k_n$

$$\mu_n^F(\{t\}) = \frac{\#(F|_{n}^{-1}(t)) }{\# V(k_n) } = \frac{\#(F|_{n}^{-1}(t))}{q^{nN}}$$

We wish to measure how far is $\mu^F_n$ from the uniform distribution on ${k_n}$.
As a measure we shall take  $$b_n=-\log_{q^n}(\max _{t,s\in k_n}(|\mu^F _n(t)-\mu^F _n(s)|)).$$
To get the asymptotic behaviour we put $$\mathcal{B}(F)= \limsup (b_n^{-1}).$$
We call $\mathcal{B}(F)$ the \emph{bias} of $F$.
\comm{\color{red}\begin{remark} If $P$ is homogeneous and the hypetsurface $X_P :=\{ v\in V-0|P(v)=0\}$ is smooth then it follows from Deligne that 
$\mathcal{B}(F)=1/2$.
\end{remark}
I expect that we need a different definition of $B(F)$. Since I am not used to the notion of $KL$-divergence I say in in my terms.

Let $$a_n=-log_{q^n}(max _{u,v\in k_n}(||\mu _n(u)|-|\mu _n(v)|)|)$$
Then $B(F)= limsup (a_n^{-1})$. So $B(F)\sim s$ iff $ |\mu _n(u)|-|\mu _n(v)|\sim q^{-s}$.}

\subsubsection{The rank of a polynomial}
\comm{\color{red}Do you define rank of $F$ over $\bar k$? If so say this explicitly.

We want to say that large rank implies small bias while Th.1 says the opposite}

Now let $G \in  \bar{k}[V]$ be homogeneous of degree $d>1$. 
We say that $G$ admits an $r$-factorization if 

$$G = \sum_{i=1}^{r} Q_i P_i$$
for $\deg Q_i, \deg P_i <d$.
The minimal $r$ such that $G$ admits an $r$-factorization is called \emph{The rank of $G$} and will be denoted by $\mathcal{R}(G)$.

More generally for $F \in k [V]$ of degree $d>1$ (not necessarily  homogeneous)  we denote 
$\mathcal{R}(F) \coloneqq \mathcal{R}(\tilde{F})$
Where $\tilde{F}$ the homogeneous part of $F$ of degree $d$ considered as polynomial over $\bar{k}$.
\section{The relationship of bias and rank} \comm{\color{red} This result is essentially in the Lovett's piece, but it make sense to give another proof. The Lovett's result does not give any explicit bounds-could our lead to such bounds?

}

We are now ready to state the main result
\begin{theorem}
Let $b\in \mathbb{R}_{>0}$, $d\in \mathbb{N}_{\geq 2}$ then there exist a constant $c(b,d)$ such that
for every  prime $p>d$, $q = p^m$ a prime power, $V$  a vector space over $\mathbb{F}_q$ and  $F \in \mathbb{F}_q[V]$  a polynomial of degree $d$ such that $\mathcal{R}(F)>c(b,d)$ we have 
$$ \mathcal{B}(F) >b$$
\end{theorem}
The passage from the rank to the bias will be achieved by analysing the codimention of the singular locus of the fibers (or their respective projectivaztions ).
We shall set up some notation.

First we denote the by $X^F$ the hypersurface in $\mathbb{P}^{N-1}$  defined by $\tilde{F}$.  For any $t \in \bar{\mathbb{F}_q}$ we denote by $Y^F_{t}$ the projective variety in  $\mathbb{P}^{N}$ 
which is the closure  of the affine hypersurface $F = t$. When $F$ is clear from the context we shall omit it from the notation.

We now have 
$$F|_n^{-1}(t) = Y^F_t({k_n}) \smallsetminus X^F({k_n}).$$

\begin{definition}
We say that a projective $X$  is $c$-regular for an integer $c$  if the singular locus of $X$ is of codimention $c$ or more.
We shall say that a polynomial $F$ is $c$-good for an integer $c$  if $X^F$ and $Y_{t}^F$ are $c$-regular for all $t \in \bar{k}$.
\end{definition}
To estimate $\#F|_n^{-1}(t)$ we now use the Weil Conjectures  proved by Deligne.

Now let  $Z$ be a projective variety of dimension $e$ defined over a field $\mathbb{F}_q$ with $\bar{Z}$  the base change to the algebraic closure and $\phi_q$ the Frobenius  map. 
By the Weil conjectures we have 
$$\# Z({k_n}) =  \sum_{i =0}^{2e}(-1)^{i}\mathrm{Tr}\left(\phi_q^n | H^i_{\acute{e}t}(\bar{Z}) \right) $$
By Weil II  the eigenvalues of $\phi_q$ on $H^i_{\acute{e}t}(\bar{Z})$ have at most absolute values $q^{i/2}$
We thus get that if $$M_Z :=  \sum_{i =0}^{2e}\dim\left(H^i_{\acute{e}t}(\bar{Z}) \right)$$

Then for every integer $m$ we have

$$\left |\# Z({k_n}) -  \sum_{i =2e-m+1}^{2e}(-1)^{i}\mathrm{Tr}\left(\phi_q^n | H^i_{\acute{e}t}(\bar{Z}) \right)\right | \leq M_Zq^{n(2e-m)} $$

Now we use the following lemma which is due to Dimca in \cite{dimca1}  in the case of singular cohomology of complex varieties . We give in the last section a proof for $\ell$-adic cohomology over a algebraic  closed field of any characteristic (Which is required for the use of the Weil conjectures ) 
\begin{lemma}\label{lem:1}
Let $Z \subset \mathbb{P}^{n+1}$  be a $c$-regular projective hyper-surface of dimension $n$
$$H^m(Z) = H^m(\mathbb{P}^{n})$$ for $2n-(c+2) \leq m \leq 2n$

\end{lemma}
Algebraic maps respects the Frobenius action. So from compering with the projective space we get

\begin{lemma}\label{lem:2}
Let $Z$ be a $c$-regular projective variety of dimension  $N$  over $\mathbb{F}_q$ then  we have 
$$\left |\# Z({k_n}) -  \sum_{i =N-\frac{c}{2}+1}^{N} q^{ni} \right | \leq M_Zq^{n(N-\frac{c}{2})} $$
\end{lemma}

Now by lemma \ref{lem:B} since all the $Y^F_{t}$ are defined by polynomials of the same  degree and number of variables we have some $M_F$ such that $$M_{X^F},M_{Y^F_t} \leq M_F$$ for all $t \in \bar{\mathbb{F}_q}$. Now by lemma \ref{lem:2} and the formula:

$$\#F|_n^{-1}(t) = \#Y^F_t({k_n}) - \#X^F({k_n})$$ 
we have: 
\begin{lemma}\label{lem:3}
Let $F$ be a $c$-good homogenous  polynomial  of  $N$ variables   over $\mathbb{F}_q$ then  we have for every $t \in {k_n}$ $$\left | \#F|_n^{-1}(t)  -  q^{n(N-1)} \right | \leq M_Fq^{n(N-\frac{c}{2})} $$ and 
$$\left |\mu_n^F(\{t\}) -\frac{1}{q^n} \right| \leq \frac{M_F}{ q^{n(\frac{c}{2}-1)}}$$
\end{lemma}
\begin{remark}\label{r:1}
Not that the proof above really proves a stronger statement namely if $F$ and $G$ are degree $d$ polynomials with the same homogenous part  and which are $c$-good , we have 
   $$\left |\mu_n^F(\{t\}) -\frac{1}{q^n} \right| \leq \frac{2M_F}{ q^{n(\frac{c}{2}-1)}}$$

\end{remark}

A conclusion of lemma \ref{lem:3} in that 

$$\max _{t,s\in k_n}\left|\mu_n^F(\{t\}) -\mu_n^F(\{s\}) \right| \leq \frac{2M_F}{ q^{n(\frac{c}{2}-1)}}$$
and thus
$$ b_n=-\log_{q^n}(\max _{t,s\in k_n}(|\mu^F _n(t)-\mu^F _n(s)|))   \geq \frac{c}{2}-1  - \frac{\log_{q}(2M_F)}{n}$$
and

$$\mathcal{B}(F)= \limsup_{n \to \infty} \frac{1}{b_n} \leq \frac{2}{c-2} $$
\begin{remark}
We thus get  a close relationship between the bias of $F$ and it's singular locus codimention.  A similar result in terms of Fourier analysis   can be found in \cite{cook}.
\end{remark}

The main result now follows from the above bound together with Lemma \ref{lem:good} in the next section.
\section{$c$-regularity and strength}
In this section we shall connect $c$-regularity  to notion of rank, one key result we shall use is Theorem A.  \cite{Hof}. 
It should be noted though that in \cite{Hof} our notion of rank is called strength.

Now given $F$ a polynomial of degree $d$ over $\mathbb{F}_q$.  For $t \in \bar{\mathbb{F}_q}$, denote by $\hat{F_t}(\bullet,Z)$ the homogenization of the  polynomial $F -t$. Here $Z$ is the extra variable of the homogenization. Note that $\hat{F_t} = \hat{F_t}(\bullet,Z)$  is the defining polynomial of the variety  $Y_t^F$ defined above.
\begin{lemma}\label{lem:5} 
For every $t \in \bar{\mathbb{F}_q}$
$$\mathrm{rank}( F) \leq \mathrm{rank}( \hat{F_t} ) \leq \mathrm{rank}(F) +1 $$
\end{lemma}
\begin{proof}
Recall that by definition $\mathcal{R} (F) = \mathcal{R} (\tilde{F}) $ where $\tilde{F}$ is the homogenous  degree $d$ part of $F$.
The first bound is now achieved by noticing that $\tilde{F} = \hat{F_t}(\bullet,0)$.
The second bound is achieved by the observation that 
$$\hat{F_t}(\bullet,Z) = \tilde{F}(\bullet) + ZG(\bullet,Z)$$
for some $G$ homogenous of degree $d-1$.
\end{proof}
\begin{remark}
Since here we are only using that $F$ and $F-t$ have the same homogenous  part taking into account remark \ref{r:1} we get that 
For  fixed $d$ and $c$. Then there exists $r,M>0$ such that for any two 
polynomial $F,G$ of degree $d$ on $n+1$ variables with the same homogenous  part  and rank $\mathcal{R}(F) =\mathcal{R}(G)  \geq r$  we have
$$\left |\#F^{-1}(0) - \# G^{-1}(0)\right |  \ll M q^{n-c}
$$
\end{remark}
The following lemma now follows immediately  from Theorem A.  \cite{Hof} and Lemma \ref{lem:5}
\begin{lemma}\label{lem:good}
Let $c,d >0$ be integers there exists a number $A(c,d)$ such that  every homogenous polynomial  $F$ of degree $d$ 
with  $\mathcal{R}(F) > A(c,d)$ we have that $F$ is $c$-good.
\end{lemma}
\section{$c$-Regularity and Cohomology }
This section is dedicated to the proof of the following lemma: 

\begin{lemma}\label{lem:A}
Let $X \subset \mathbb{P}^{N+1}$  be a $c$-regular projective hyper-surface of dimension $n$
$$H^m(X) = H^m(\mathbb{P}^{N})$$ for $2N-(c+2) \leq m \leq 2N$

\end{lemma}

\begin{proof}
Consider the constant  $\ell$-adic sheaf $\QQ_\ell$ on $\mathbb{A}^{N+2}$. The sheaf $\QQ_\ell[N+1]$ is perverse on $\mathbb{A}^{N+2}$  
Let $F:\mathbb{A}^{N+2} \to \mathbb{A}^1$ be an \textbf{homogenous} polynomial. Let $\phi_F$ be the functor of vanishing cycles attached to $F$ 
We get that $\phi_F(\QQ_\ell[N+2])[-1] = \phi_F(\QQ_\ell)[N+1]$ is perverse on $F^{-1}(0)$ and concentrated on the critical  locus of $F$, that is the singular locus
of $F^{-1}(0)$ by assumption the  singular locus is of dimension $s:= N+1-c$ so by perversity the stalk of $\phi_F(\QQ_\ell)[N+1]$ at any point  has noN-zero homology only in degrees $[0,s]$ consider now the short exact sequence of  perverse sheafs on $F^{-1}(0)$

$$0\to  \QQ_\ell[N+1]\to \psi_F(\QQ_\ell)[N+1] \to \phi_F(\QQ_{\ell})[N+1] \to 0$$
Here $\psi_F$ is the nearby cycle functor. We conclude that the stalks of $\psi_F(\QQ_\ell)$ are concentrated in degrees $0$ and $[-N-1,-N+s-1]$

Now let $i:F^{-1}(0) \to \mathbb{A}^{N+2}$ be the close embedding and $j$ the  embedding of the open  complement.
We have the Wang cofiber sequence 

$$i^*j_*j^*\QQ_\ell \to \psi_F(\QQ_\ell) \xrightarrow{T-id} \psi_F(\QQ_\ell) $$

So the stalks of  $i^*j_*j^*\QQ_\ell$  are concentrated at degrees   $[-1,0]$ and $[-N-2,-N+s-1]$ in particular let $p:\mathbb{A}^0 \to \mathbb{A}^{N+2}$ be the inclusion of the origin. We get that $p^* j_*\QQ_{\ell}$ is concentrated at degrees   $[-1,0]$ and $[-N-2,-N+s-1]$. The following lemma is  a special case of 
Lemma 6.1 in \cite{Verd}.
\begin{lemma}\label{lem:verd}
Let $$J:\mathbb{A}^{N+2} \smallsetminus {O} \to \mathbb{A}^{N+2}$$ be the embedding and let $A$ be a $\mathbb{G}_m$ invariant   sheaf on $\mathbb{A}^{N+2} \smallsetminus {O} $ let $\pi:\mathbb{A}^{N+2} \to\mathbb{A}^{0} $ be the projection then 
$$\pi_*(J_{!}(A)) = 0$$
\end{lemma}

\begin{corollary}
The map 
$$ \pi_*j_*\QQ_{\ell}\to \pi_*p_*p^*j_*\QQ_{\ell}  \cong p^*j_*\QQ_{\ell} $$
is an equivalence.
\end{corollary}
\begin{proof}
We have a cofiber sequence 
$$ J_!J^!j_*\QQ_{\ell} \to j_*\QQ_{\ell} \to p_*p^*j_*\QQ_{\ell}$$
and therefore it is enough to show that $\pi_*J_!J^!j_*\QQ_{\ell}   = 0$ by \ref{lem:verd} we reduced to show that $J^!j_*\QQ_{\ell} = \kappa_*\QQ_{\ell}$ is $\mathbb{G}_m$ invariant where 
$$\kappa \colon F^{-1}(\mathbb{A}^1\smallsetminus \{0\}) \to \mathbb{A}^{N+2} \smallsetminus \{0\}$$
is the embedding. This is immediate from the homogeneity of $F$.
\end{proof}

Now let $$\hat{\pi}\colon F^{-1}(\mathbb{A}^1\smallsetminus \{0\}) \to \mathbb{A}^{0}$$ be unique map. We have that $\hat{\pi}_*\hat{\pi}^*\QQ_{\ell} $   is concentrated in degrees  $[-1,0]$ and $[-N-2,-N+s-1]$. So $H^{m}(F^{-1}(\mathbb{A}^1\smallsetminus \{0\})) = 0$ for $2 \leq m \leq N-s$.

Now let $$\iota \colon F^{-1}(0) \smallsetminus \{0\} \to \mathbb{A}^{N+2}\smallsetminus \{0\}$$ be the emmebeding. $\mathbb{A}^{N+1} \smallsetminus \{0\}$ behaves cohomoligcally  like  a sphere of dimention $2N+3$ so we get by Alexander duality (for $\iota$ and $\kappa$ ) that $$H^{m}(F^{-1}(0)  \smallsetminus \{O\} = 0$$ for  $$  2N+3-c = N+2 +s \leq m\leq 2N.$$

Let us now denote $S:= \mathbb{A}^{N+2} \smallsetminus \{O\}$ and $K = F^{-1}(0) \smallsetminus \{O\}$ and 
consider  the map of $\mathbb{G}_m$ bundles: 

$$
\xymatrix{
K \ar[r]\ar[d]  &S \ar[d]\\
X  \ar[r] & \mathbb{P}^{N}
}
$$

By taking the associated  Gysin sequences we get:

$$\xymatrix{
\empty\ar[r] & H^{m+1}(S) \ar[r]\ar[d]&  H^{m}(\mathbb{P}^{N+1}) \ar[r]^{\psi}\ar[d] &  H^{m+2}(\mathbb{P}^{N+1}) \ar[r]\ar[d]&  H^{m+2}(S) \ar[r]\ar[d]& \empty \\
\empty\ar[r] & H^{m+1}(K) \ar[r] &  H^{m}(X)  \ar[r]^{\psi_V} &  H^{m+2}(X) \ar[r] &  H^{m+2}(K) \ar[r] & \empty 
}$$

Now by the above computation   $$H^{2N}(K) =0$$    and $$H^{2N+1}(X) = 0$$ since $X$ is of dimension $N$ so $H^{2N-1}(X) =0$. Now since $H^{2N-2}(K) =0$ by the above computation we get that $H^{2N-3}(X) = 0$ as well.
We proceed  by decreasing induction and get the result for odd $m$.
Similarly the result for $H^{2N}(X)$ follows from $V$ being irreducible (Note that of $c <2$ the statement of the lemma is vacuous)

\end{proof}
This additional lemma is useful for achieving the bounds in the proof.  
\begin{lemma}\label{lem:B}

Let $k$ be a field  and  $N$ and $d$ be integers. There exists a universal bound $M(N,d)$ such that 

For every non zero  homogenous  polynomial $F \in k[x_0,...,x_{N}]$ the zero locus of $F$ , the variety  $V_F \subset \mathbb{P}^{N}$ satisfies  

$$M_{V_F}:=  \sum_{i =0}^{2e}\dim\left(H^i_{\acute{e}t}(\bar{V}_F) \right) \leq M(N,d)$$

\end{lemma}
\begin{proof}
Let $W_{N,d}:=  k^{d}[x_0,...,x_{N}]$ be the vector space of degree $d$ homogenous polynomials  (note that $\dim W_{N,d}   ={{N+d} \choose{d}}$). 
Denote by $$X_{N,F}  \subset \mathbb{P}^{N} \times \mathbb{P}(W_{N,d})$$
The variety of points $(v,F)$ such that $F(v) = 0$
We have a proper projection map 
$$ \pi\colon X_{N,F} \to  \mathbb{P}(W_{N,d}).$$
The sheaf $C = \pi_*\QQ_{\ell}$ is constructible on $\mathbb{P}(W_{N,d})$. by the proper base change theorem the stalk of $C$ over a polynomial $F \in \mathbb{P}(W_{N,d})(k) $ is the $\ell$-adic cohomology complex of $V_F$. by the boundedness properties of constructible sheaves we are done. 
\end{proof}

\bibliographystyle{amsalpha}
\bibliography{biblio}

\end{document}